\documentclass[12pt]{article}
\usepackage[a4paper]{geometry}
\usepackage{amsthm}

\usepackage{amssymb}
\usepackage{amsmath}
\usepackage{eufrak}

\newtheorem{theorem}{Theorem}
\newtheorem{defn}[theorem]{Definition}
\newtheorem{remark}[theorem]{Remark}

\newtheorem{example}[theorem]{Example}
\newtheorem{corollary}[theorem]{Corollary}

\newtheorem{definition}[theorem]{Definition}

\newtheorem{proposition}[theorem]{Proposition}
\newtheorem{notation}[theorem]{Notation}
\newtheorem{prop}[theorem]{Proposition}
\begin{document}
\def\F{{\mathbb F}}
\title{On flat deformations and their applications}
\author{ Agata Smoktunowicz}
\date{ }
\maketitle

\begin{abstract} 

We say that a formal deformation from an algebra $N$
 to an algebra $A$ 
is strongly flat if for every real number $\epsilon $ there is a real number $0<s<\epsilon$ such that this deformation specialised at $t=s$ gives an algebra isomorphic to  $A$. It is shown that all semisimple algebras which can be obtained as a specialisation of  such a deformation   are isomorphic.

   We  also show that every strongly flat deformation  $\mathcal N=N\{t\}$  from a finite-dimensional $\mathbb C$-algebra $N$ to a semisimple $\mathbb C$-algebra $A$  
  specialised at $t=s$ for all sufficiently small real  numbers $s>0$ gives an  algebra isomorphic to $A$.  A remark by Joachim Jelisiejew is also included  which allows us to obtain this result as an application of Gabriel's theorem \cite{Gabriel}. 

 We also give a characterisation of semisimple algebras $A$ to which  a given algebra $N$ cannot be  strongly flatly  deformed.   This gives a partial answer to a question of Michael Wemyss's  on  Acons \cite{Wemyssprivatecomm} as well as
 a  partial answer to question 6.5 from \cite{DDS}.

\end{abstract}

\section{Introduction}
  
 This paper is motivated by some questions asked by Michael Wemyss about deformations of  contraction algebras, in particular his question of whether  a given contraction algebra can be flatly deformed to only one semisimple algebra. We also provide some criteria for finding semisimple algebras $A$ to which a given algebra cannot be deformed. 
 
  Contraction algebras  have  applications in algebraic geometry in particular in resolutions of singularities and in connection to geometric invariants. Deformations of contraction algebras into semisimple algebras are related to  Gopakumar-Vafa invariants.  Contraction algebras were introduced by Donovan and Wemyss in \cite{DW} and they are certain factors of the maximal modification algebras (MMAs) developed by Iyama and  Wemyss in \cite{I}.

We only consider deformations $N\{t\}$ of $N$ in which the image of the identity element in $N$ acts as the identity element in $N\{t\}$. We will denote $N\{t\}$ by $\mathcal N.$

We say that a formal deformation from an algebra $N$ to an  algebra $A$ is strongly flat if for every real number $\epsilon $ there is a real number $0<s<\epsilon $ such that this deformation specialised at $t=s$ gives an algebra isomorphic to $A$.

  Our first result is: 
\begin{proposition}\label{main} Let $\mathcal N$ be a strongly flat deformation 
from a finite-dimensional $\mathbb C$-algebra $N$ to a semisimple $\mathbb C$-algebra $A$ of the same dimension. 
Then the deformation $\mathcal N$ specialised at  $t=s$ for all sufficiently small real  numbers $s>0$ gives an  algebra isomorphic to $A$.
 In particular if $\mathcal N$ is also a  strongly flat deformation from $N$ to a semisimple algebra $B$, then $A$ and $B$ are isomorphic $\mathbb C$-algebras.  
\end{proposition}
Next we obtain some criteria for when a given algebra $N$ cannot be deformed flatly to a given semisimple algebra $A$ (Proposition \ref{ce}, Corollary \ref{cde}). We also apply these results to show that some contraction algebras cannot    be deformed to 
 some semisimple algebras thereby partially answering a question by Michael Wemyss \cite{Wemyssprivatecomm} (Example \ref{ef}).

  Let $N, A$ be two finite-dimensional $\mathbb C$-algebras of the same dimension $n$ with $A$ semisimple.  Notice that there are only finitely many non-isomorphic semisimple $\mathbb C$-algebras of dimension $n$. It follows that  if $\mathcal N$ is a flat deformation of $N$ which deforms to $A$ at $t=s$ then we have the following possibilities: (i) $\mathcal N$ is a strongly flat deformation from $N$ into $A'$, for some semisimple $\mathbb C$-algebra $A'$, where $A'$ is uniquely determined by $\mathcal N$,  
(ii) for all sufficiently small real numbers $s>0$ the algebra $\mathcal N$ specialised at $t=s$ is not semisimple, and hence it has a non-zero nilpotent ideal.

The following interesting remark about case (ii) by Joachim Jelisiejew was communicated after the first version of this paper was uploaded to arxiv.org, and he has agreed  it can be included here. 

{\em Remark (provided by Joachim Jelisiejew).} `` I believe that the case (ii) does not hold, as a deformation does not exist by an old result of Gabriel, [Corollary 2.5, page 142, \cite{Gabriel}].

Gabriel proves that the orbit of an algebra A is Zariski-open whenever $H^2(A, A) = 0$. In your setup, we work over complex numbers, hence semisimple algebras are automatically separable, so that the Hochschild cohomology $H^{>0}(A, -)$ vanishes for them, in particular $H^2(A, A)$ vanishes.
A deformation in the sense of your article yields a map from a small segment $[0, eps)$ to the space ${\mathrm Alg}_n$ of Gabriel. The image of this map touches the orbit of $A$, hence, by Zariski-openness, all but finitely many points of $[0, eps)$ lie in the orbit of $A$. In particular, sufficiently small $s > 0$ is mapped to the orbit of $A$, so that it yields a semisimple algebra. The same argument seems to allow one to deduce your Proposition $1$.''

 Note that our proof of Proposition  \ref{main} differs from the  proof mentioned by Joachim Jelisiejew. We present it so as to have an alternative proof of Proposition \ref{main} and  because it contains some  facts which are also useful for investigating Acons and are later used in paper \cite{Wemyss3} as supporting lemmas. 
 Interesting open questions on deformations of algebras appear in \cite{JJ, JJ3, JJ4, JJ5}.
  
 By combining the above suggestion from Joachim Jelisiejew  with our methods from the proof of Proposition \ref{main} we obtain the following theorem:
\begin{theorem}\label{main2} Let $N$, $A$, $B$ be finite-dimensional $\mathbb C$-algebras of the same dimension, and suppose that $A$ and $B$ are semisimple $\mathbb C$-algebras. Let $\mathcal N$ be a  formal deformation of $N$. Suppose that $\mathcal N$ is a flat deformation, in the sense of Definition \ref{flat}, 
from $N$ to  $A$. Suppose  that    $\mathcal N$ is also flat deformation, in the sense of Definition \ref{flat}, 
from  $N$ to $B$. Then $A$ and $B$ are isomorphic  $\mathbb C$-algebras. Moreover, $\mathcal N$ is a strongly flat deformation from $N$ to $A$.
\end{theorem}

 Interesting applications of deformation theory in noncommutative algebra, representation theory of algebras and Hopf-Algebras  were considered in \cite{IG, IG2, HV, Shepler}. 

\section{Background information}

By $\mathbb C$ we denote the field of complex numbers.  
We recall the definition of the formal deformation of $A$ from \cite{SWbook}.
 
\begin{defn}\label{SWbook} A formal deformation $(A_{t}, \circ, +)$ of a $\mathbb C$-algebra $(A, \cdot, +)$ is an associative ${\mathbb C}$-bilinear multiplication $\circ $ on the $\mathbb C\{t\}$-module $A\{t\}$ (where $A\{t\}$ is the set of power series with coefficients from $A$).

Recall that $A\{t\}$ has natural structure of $\mathbb C\{t\}$-module given by 
\[t^{j}\sum_{i=0}^{\infty }a_{i}t^{i}= \sum_{i=0}^{\infty }a_{i}t^{i+j}.\]
\end{defn}

 $(A_{t}, \circ, +)$ is a deformation of {\em polynomial type} of a $\mathbb C$-algebra $(A, \cdot, +)$ if for all $a,b\in A$, $a\circ b\in A[t]$, where $A[t]$ denotes the polynomial ring in variable $t$ and coefficients in $A$.  

$ $

In this paper we will only consider the case when $A$ is a finite-dimensional $\mathbb C$-algebra. Moreover, we only consider  deformations $A\{t\}$ of $A$ in which the image of the identity element in $A$ acts as the identity element in $A\{t\}$. All algebras considered in this paper are associative and usually noncommutative.

 Let $b_{1}, \ldots , b_{n}\in A$ be a basis of $A$ as a linear vector space over field $\mathbb C$. Note that every element of $A\{t\}$ can be written in the form
\[\sum_{i=1}^{n}b_{i}f_{i}(t),\]
 for some $f_{i}\in \mathbb C\{t\}$.
 
Then, we have a well defined multiplication on $A\{t\}$ given by
\[(\sum_{i=1}^{n}b_{i}f_{i}(t))\circ (\sum_{j=1}^{n}b_{j}f_{j}(t))=\sum_{i=1}^{n}\sum_{j=1}^{n}b_{i}\circ b_{j}f_{i}(t)f_{j}(t).\]

Notice that this multiplication is associative, since the multiplication $b_{i}\circ b_{j}$ is associative, so $(A\{t\}, \circ, +)$ is an associative algebra. We assume that $(A, \circ , +)$  has an identity element $1_{A}$, so we can identify elements of $\mathbb C\{t\}$ with elements of the subring  ${\mathbb C}\{t\}\cdot 1_{A}\subseteq A$, and we have $t^{i}\circ a=at^{i}=a\circ t^{i}$ for all $i$ and all $a\in A$. In particular $t$ is not a zero divisor in the algebra $(A\{t\}, \circ, +)$. 
 Similarly if $f(t)\in \mathbb C\{t\}$ then $f(t) $ (which we identify with $f(t)\cdot 1_{A}$) is not a zero divisor of $A\{t\}$.

  We recall a definition of a flat deformation following Feigin and Odesskii  \cite{FO} (we added an assumption on convergence of some power-series to their definition):
 \begin{definition}\label{flat}
 An algebra $A$ can be flatly deformed to an algebra $B$ if there is a formal deformation $A\{ t\}$ of $A$ such that:
\begin{enumerate}

\item For some $c\in \mathbb C$  the specialisation $A_{s}$ of $A\{t\}$  at $t=s$ is well defined for all $s\in (0,c)$ (so all power series $g_{i,j,l}(t)$ appearing in Definition \ref{76} are convergent for $t$ from the  interval $(0,c)$).  Moreover, 
 $A_{t_{0}}$ is isomorphic to $B$ for some $t_{0}\in (0, c)$.  
 Notice that $A_{0}$ is equal to $A$, since $A\{t\}$ is a formal deformation of $A$. 
 
\item The dimension of $A\{c\}$ for $0\leq c \leq t_{0}$ is the same as the dimension of $A$.

\end{enumerate}
\end{definition}

 We now recall a definition of a deformation used in algebraic geometry (provided by Joachim Jelisiejew):
\begin{definition}\label{Deformation} A deformation of a $\mathbb C$-algebra $N$ over ${\mathbb C}\{t\}$ is a ${\mathbb C}\{t\}$-algebra  $\mathcal{N}$ such that $ \mathcal{N}$ is a free ${\mathbb C}\{t\}$ module and  the algebra $\mathcal{N} / t \mathcal{N}$ is isomorphic with $ N.$

 A deformation of algebra $N$ to algebra $A$ over ${\mathbb C}\{t\}$  is a deformation over ${\mathbb C}\{t\}$ as above  and such that $ \mathcal{N} \otimes_{{\mathbb C}\{t\}} {\mathbb C}\{\{t\}\}$ is isomorphic with $ A\{\{t\}\}=A\otimes _{{\mathbb C}} {\mathbb C}\{\{t\}\}$,  where ${\mathbb C}\{\{t\}\}$ denotes the algebraic closure of ${\mathbb C}\{t\}$ (so ${\mathbb C}\{\{t\}\}$ is an algebraically closed field), and $\otimes $ denotes the tensor product  (over a central subalgebra so it is well defined).
\end{definition} 
   We introduce a definition of a strongly flat deformation. 

\begin{definition}\label{stronglyflat} Let $A$ and $B$ be finite-dimensional $\mathbb C$-algebras of the same dimension. 
 An algebra $A$ can be strongly flatly deformed to an algebra $B$ if there is a formal deformation $A\{ t\}$ of $A$ such that:
\begin{enumerate}

\item For some $c\in \mathbb R$  the specialisation $A_{s}$ of $A\{t\}$  at $t=s$ is well defined for all  $s\in [0,c)$ (so all power series $g_{i,j,l}(t)$ appearing in Definition \ref{76} are convergent for $t$ from the  interval $[0,c)$).  Moreover, 
 $A\{t_{0}\}$ is isomorphic to $B$ for some $t_{0}\in (0, c)$.

\item  Moreover, for every real number $\epsilon>0$
 there is $0<s<\epsilon $ such that $A_{s}$ is isomorphic to $B$.
\end{enumerate}
  
 Notice that $A\{0\}$ is equal to $A$, since $A\{t\}$ is a formal deformation of $A$.
\end{definition}
 Notice that if $B$ is a semisimple algebra then Definition \ref{stronglyflat} implies Definition \ref{flat}. So a strongly flat deformation from $A$ to $B$ with $B$ semisimple is also a flat deformation in the sense of the above Definition  \ref{flat}. 
 Notice also that Definition \ref{stronglyflat} implies Definition \ref{Deformation}.  
 So strongly flat deformation from $A$ to $B$ with $B$ semisimple is also a  deformation from $A$ to  $B$  in the sense of the above Definition \ref{Deformation}.

We will now recall some definitions from \cite{DDS}, as we will be using the same notation.

\begin{defn}
     Let $S$ be a subset of a $k$-algebra $A$ for a commutative unital ring $k$. Then $S$ is a \textbf{generating set} if all elements of $A$ can be written as a $k$-linear sum of products of elements of $S$ using the operations in $A$.
\end{defn}

\begin{notation}
    We will denote by $\mathbb C\langle x, y \rangle$ the polynomial ring over $\mathbb C$ in two non-commuting variables, $x$ and $y$.
\end{notation}

\begin{notation}
    We will denote by $\mathbb C\langle x, y \rangle \{t\}$ the power series ring in variable $t$ with coefficients from the polynomial ring in two non-commuting variables $x$ and $y$, with coefficients from $\mathbb C$. By $\mathbb C\langle x,y \rangle [t]$ we will  denote the polynomial ring in variable $t$ over the ring  $\mathbb C\langle x,y \rangle $.  By $\mathbb C\{t\}$ we denote the power series ring in variable $t$ and coefficients from $\mathbb C$. By ${\mathbb C}[t]$ we denote the polynomial ring in variable $t$. 
\end{notation}

\begin{notation}
   Let $R$ be a ring and $I$ be an ideal of $R$, then elements of the factor ring $R/I$ will be denoted as $r+I$, where $r\in R$.
 Notice that $r+I=s+I$ for $r,s\in R$ if and only if $r-s\in I$.
\end{notation}

\subsection{Polynomial identities}
 We start with a definition.

\begin{definition}\label{K} Let $s>0$ be a real number and let ${\bar {\mathbb C}}\{t\}$ denote the ring of power series from ${\mathbb C}\{t\}$ which are convergent at some interval containing zero.
 By    $K$ we denote the field of fractions of  ${\bar {\mathbb C}}\{t\}$. 
\end{definition}

 \begin{definition}\label{76} Let $N$ be a $\mathbb C$-algebra,
 let $d_{1}, \ldots , d_{n}$ be basis of $A$ as a linear space over $\mathbb C$. Let $\circ _{t}$ give the formal deformation of $N$ where 
\[d_{i}\circ _{t}d_{j}=\sum_{l=0}^{\infty }u_{l}(d_{i}, d_{j})t^{l}.\]

 We call the obtained ${\mathbb C}\{t\}$-algebra with basis $d_{1}, \ldots , d_{n}$ and multiplication $\circ _{t}$, $\mathcal N$. 
 We call the obtained ${\bar {\mathbb C}}\{t\}$-algebra with basis $d_{1}, \ldots , d_{n}$ and multiplication $\circ _{t}$, as ${\bar {\mathcal N}}$. 
By $E$ we denote the $K$-algebra with a basis (as a linear space) $d_{1}, \ldots , d_{n}$ and with the multiplication 
\[d_{i}\circ _{t}d_{j}=\sum_{l=0}^{\infty }u_{l}(d_{i}, d_{j})t^{l},\]

$ $

 Note that this extension  is well defined because ${\bar {\mathcal N}}$ is a ${\bar {\mathbb C}}\{t\}$-module.
Another way to see it is to embed ${\bar {\mathcal N}}$
 into an algebra of matrices with entries from ${\bar {\mathbb C}}\{t\}$ by the usual embedding when $x\rightarrow L_{x}$ where $L_{x}(y)=x\cdot y$ for $x,y\in {\bar {\mathcal N}}$, then this ring of  matrices is a natural subring of the ring of matrices with entries from $K$- the field of fractions of ${\bar {\mathbb C}}\{t\}$. 

$ $

 Note that \[d_{i}\circ _{t}d_{j}=\sum_{l=0}^{n }d_{l}g_{i,j,l}(t),\]
 for some $g_{i,j,l}(t)\in {\mathbb C}\{t\}$. Note that if $\mathcal N$ gives a flat deformation for $t\in [0,e)$ for some real number $e$ then $g_{i,j,l}(t)$ are convergent power series at $t\in [0,c)$ for some real number $0<c$ (and hence strongly convergent for $t\in [0,c)$).
\end{definition}

 We start with the following result:

\begin{prop}\label{semisimple} Let $N, A$ be two finite-dimensional $\mathbb C$-algebras, of the same dimension. 
 Let $E, \mathcal N, d_{1}, \ldots , d_{n}$ be as in Definition \ref{76}. 
Suppose that $\mathcal N$ gives a strongly flat deformation from $N$ to $A$. 
  Suppose that $A$ is a semisimple $\mathbb C$-algebra, then $E$ is a semisimple $K$-algebra of the same dimension.
\end{prop}
\begin{proof} We will use a similar idea as in the classical papers in deformation theory. 
 Observe first that $E$ has no nonzero nilpotent ideals, as if $J$ is a nilpotent ideal in $E$ then there is 
 $0\neq d_{1}h_{1}(t)+\ldots +d_{n}h_{n}(t)$ which is a nonzero  nilpotent element, and 
the ideal generated by this element in $E$ is nilpotent, where $h_{1}(t), \ldots , h_{n}(t)\in {\bar C}\{t\}$. 
 Let $s>0$ be such that at $t=s$, $\mathcal N$ deforms to $A$ (which means 
that  $\mathcal N_{s}$ is isomorphic to $A$) and that $h_{1}(t), \ldots , h_{n}(t)$ are absolutely
 convergent for $t\in [0,2s)$. We can find such $s$ because we can take arbitrarily 
small $s>0$ (since $\mathcal N$ is {\em strongly} flat). Because we can take arbitrarily small $s$  we can assume that  
  $h_{1}(s)\neq 0$. Notice that $h_{i}$ is convergent because of the definition of $E$, as $E$ only
 involves power series from $K$. Moreover,
 $0\neq d_{1}h_{1}(t)+\ldots +d_{n}h_{n}(t)$ generates a nilpotent ideal in
 the ring of power series which are absolutely convergent at the interval $[0,2s)$.  
 It follows that $d_{1}h_{1}(s)+\ldots +d_{n}h_{n}(s)$ generates nonzero 
nilpotent ideal in the specialisation of $\mathcal N$ at $t=s$. We have 
obtained a contradiction, since this specialisation at $t=s$ is isomorphic to $A$ and $A$ is a semisimple $\mathbb C$-algebra, so it has no nilpotent ideals. 

Therefore $E$ has no nonzero nilpotent ideals, and since $E$ is a finite-dimensonal algebra over a field 
  it follows that $E$ is semisimple.
\end{proof}

\begin{definition}\label{12345}
Let $R$ be a finite-dimensional  semisimple algebra over a field $K.$ 

Then, by the Artin-Wedderburn theorem $R$ is isomorphic (for some $n'$)  to
\[\oplus _{i=1}^{n'} M_{n_{i}}(D_{i}(R))  ,\]
 where $D_{i}(R)$ is a finite-dimensional skew-field extension of $K$.  
 By $e_{j}(R)$ we denote \[\oplus _{\{i=1, \ldots, n': n_{i}=j\}}M_{n_{i}}(D_{i}(R)).\]

\end{definition}

 \begin{remark}\label{commutative}
Let $N, A$ be two finite-dimensional $\mathbb C$-algebras, of the same dimension. 
 Let $E, \mathcal N, D_{i}$ be as in Definition \ref{76} and Proposition \ref{semisimple}.   Suppose that $\mathcal N$ gives a strongly flat deformation from $N$ to $A$.  
  Then for every $j$  the division ring $D_{j}(E)$ is commutative.
\end{remark}
\begin{proof} 
 It follows from the known fact that all finite-dimensional division algebras over $K$ are commutative, where $K$ is as in Definition \ref{K}. 

{\em Proof $1$.}:
{\em Part $1$.} {For this part we quote  Uzi Vishne who kindly advised me with the proof.}: ``The key notions here are Tsen's theorem and $C_{1}$ fields (also called  ``quasi-algebraic''). See \cite{Serre}, Section II.3 (subsection II. 3.2 is about $C_{1}$).  
  The property that all finite-dimensional skew fields over $F$ are commutative is the same as $Br(F')=1$ for all finite-dimensional field extensions of $F$. This is exactly the statement that $F$ has cohomological dimension not exceeding one (denoted $cd(F)\leq 1$), also defined in \cite{Serre}. In 1933 Tsen proved that the field of rational functions ${\mathbb C}(t)$ is $C_{1}$ (where $\mathbb C$ is the field of complex numbers), which implies $cd({\mathbb C}(t))\leq 1$.  The fact that ${ \mathbb C}((t))$ is $C_{1}$  also follows (see \cite{GS}, Theorem $6.2.1$.).''
 
{\em Part $2$.} Note also that in \cite{Lang}, Theorem $12$ states ``The following fields and their algebraic  extensions are $C_{1}$: The convergent power series over an algebraically closed valuated constant field.''
 Therefore $K$ is $C_{1}$, hence all finite-dimensional division algebras over $K$ are commutative.

 {\em  Proof $2$.} An  alternative proof that $K$ is $C^{1}$  using Henselian rings:
 
 Recall  Definition  $2.18$ and an example included just after in this definition from \cite{reference27}:``  A discrete valuation ring is a Noetherian local domain whose maximal ideal is
 principal and different from $(0)$.
  The ring of algebraic power series $k\langle x\rangle $ or convergent power series $k\{x\}$ in one variable $x$ are
 Henselian discrete valuation rings.''

 The result that $K$ is $C_{1}$ also follows from \cite{reference2}, page 566,  Note 1: ``The same type of argument yields a short proof of Lang's theorem that
 if $R$ is a Henselian discrete valuation ring with algebraically closed residue field, such
 that $K^{*}$  is separable over $K$, then $R$ is $C^1$.''
\end{proof}

 We will now prove the following result: 

\begin{prop}\label{polynomialidentity} Let $N, A$ be two finite-dimensional $\mathbb C$-algebras, of the same dimension. 
 Let $E, \mathcal N, D_{i}$ be as in Definition \ref{76} and Proposition \ref{semisimple}.   Suppose that $\mathcal N$ gives a strongly flat deformation from $N$ to $A$.   
  Then for every $j$ we have \[dim_{\mathbb C}e_{j}(A)=dim _{K}e_{j}(E).\]
\end{prop}
\begin{proof} Let $d_{1}, \ldots , d_{n}$ be a basis of $A$, then because $\mathcal N$ is a formal deformation of $N$, then $d_{1}, \ldots , d_{n}$ span $E$ as $K$-algebra, and they also span $N_{s}$ ($\mathcal N$ specialised at $t=s$).

 {\em Part 1.} We will first show that  
\[dim_{\mathbb C}e_{j}(A)\geq dim _{K}e_{j}(E).\]

It is known that the polynomial identity for $m\times m$ matrix algebra $R$ over a field $K$ (of characteristic zero) is:
\[s_{2m}(x_{1}, \ldots , x_{2m})=  \sum_{\sigma}(-1)^{\sigma} x_{\sigma (1)}\cdot _{R}  \cdots \cdot _{R} x_{\sigma (2m)}, \]
  where $x_{1}, \ldots , x_{2m}$ are elements of the given algebra $R$, and $\cdot _{R}$ is the multiplication in algebra $R$.
It is known that all matrices of dimension at most $m$ over a field satisfy this identity, but matrices of dimension $m+1$ and larger dimension don't satisfy it \cite{Rowen}. 

$ $

We then have  in the algebra $A$:
\[s_{2m}(x_{1}, \ldots , x_{2m})=  \sum_{\sigma}(-1)^{\sigma}x_{\sigma (1)} \cdots \cdot  x_{\sigma (2m)}, \]
 where $\cdot $ is the multiplication in $A$. Let $I_{m}\subseteq A$ denote set of evaluations   
$s_{2m}(x_{1}, \ldots , x_{2m})$ with $x_{1}, \ldots , x_{2m}\in \{d_{1}, \ldots , d_{n}\}$.

 Let $I_{m}'(A)$ be the ideal generated in the algebra $A$ by elements 
$s_{2m}(x_{1}, \ldots , x_{2m})$ for all possible choices of elements $x_{1}, \ldots , x_{2m}\in \{d_{1}, \ldots , d_{n}\}$.

$ $

Similarly, in $E$ we have 
\[s_{2m}(x_{1}, \ldots , x_{2m})=  \sum_{\sigma}(-1)^{\sigma} x_{\sigma (1)}\circ_{t} \ldots \circ _{t} x_{\sigma (2m)}, \]
  where $x_{1}, \ldots , x_{2m}$ are elements of the given algebra $E$, and $\circ _{t}$ is the multiplication in algebra $E$.

 Let $I_{m}(E)\subseteq E$ consist of  elements 
$s_{2m}(x_{1}, \ldots , x_{2m})$ for all possible choices of elements $x_{1}, \ldots , x_{2m}\in \{d_{1}, \ldots , d_{n}\}$.
Let $I_{m}'(E)$ be the ideal in $E$ generated by elements from $I_{m}(E)$.

Fix a natural number $j$. 
 Suppose that $dim_{K} E-dim_{K} I_{j}'(E) = \beta $.

As the enumeration of elements $d_{i}$ does not matter, we can assume that 
 \[E=span _{K}\sum_{i=1}^{\beta }d_{i}K+I_{j}'(E).\]
Therefore, for some $ g_{1}(t), \ldots , g_{\beta }(t)\in K$ we have 
\[d_{\beta +1}\in \sum_{i=1}^{\beta } g_{i}(t)d_{i}+I_{j}'(E).\]

Therefore there are $h(t), h_{1}(t), \ldots , h_{\beta }(t)\in {\bar {\mathbb C}}\{t\}$
 with $h(t)\neq 0$ and $c(t)\in   {\bar {\mathbb C}}\{t\}I_{j}'(E)$ 
 such that 
\[h(t)d_{\beta +1}= \sum_{i=1}^{\beta } h_{i}(t)d_{i}+c(t).\]

 We can assume that $c({t})$ is a linear combination over  ${\bar {\mathbb C}}\{t\}$ of elements from the set $\{d_{1}, \ldots , d_{n}\}$ (because if needed we can adjust $h(t)$).
 So we have \[c(t)=h(t)d_{\beta +1}-\sum_{i=1}^{\beta } h_{i}(t)d_{i}.\]

This is  relation which holds in the ${\bar {\mathbb C}}\{t\}$-algebra $\bar {\mathcal N}$. 
 Recall that $\bar {\mathcal N}$ specialised at $t=s$ gives multiplication in $N_{s}$-the algebra spanned by $d_{1}, \ldots , d_{n}$.

Let $0<s'<s$.  By specialising $c(t)$ in $t=s'$ we get relation in $N_{s'}$. 
\[c(s')=h(s')d_{\beta +1}-\sum_{i=1}^{\beta } h_{i}(s')d_{i},\]
 where $c(s')$ is obtained from $c(t)$ by specialising at each place
 the multiplication for $t=s'$, for example if $c(t)=(d_{1}\circ _{t}d_{2}-d_{2}\circ _{t}d_{1})\circ _{t}d_{3}$, 
then $c(s')=(d_{1}\circ _{s'}d_{2}-d_{2}\circ _{s'}d_{1})\circ _{s'}d_{3}$,
 Note, that $c(s')\in I_{j}'(A)$, by the construction. 

$ $

Because $h_{i}(t), h(t)$ are absolutely convergent for $0< t<s$  
 there is $e>0$ such that for  all $0<s'<e$ we have  $h(s')\neq 0$. By the assumptions we can find $s\in (0, e)$ such that  $\mathcal N$ specialised at $t=s$ gives an algebra $N_{s}$  isomorphic to $A$, therefore,
 
\[d_{\beta +1}\in \sum_{i=1}^{\beta }{ \mathbb C}d_{i}+I_{j}'( N_{s})\]

We can apply the same argument and obtain that for $i>0$ we have 
\[d_{\beta +i}\in \sum_{i=1}^{\beta }{ \mathbb C}d_{i}+I_{j}'( N_{s})\]

Therefore $dim_{\mathbb C}I_{j}'(A)\geq n-\beta $. Since 
 by assumption $dim_{K}I_{j}'(E)=n-\beta $ we obtain 
$dim_{\mathbb C}I_{j}'(N_{s})\geq dim_{K}I_{j}'(E)$ for every $j$.
 Because $N_{s}$ is isomorphic to $A$ we obtain $dim _{\mathbb C}(A)=dim_{\mathbb C}I_{j}'(N_{s})\geq dim_{K}I_{j}'(E)$ for every $j$. 
 Observe that matrix rings over 
 fields are simple algebras over these fields, therefore, \[I_{j}'(E)=\oplus_{i>j}e_{i}(E), I_{j}'(A)=\oplus_{i>j}e_{i}(A).\]

{\em Part 2.} We will now show that $dim_{\mathbb C}I_{j}'(A)\leq dim_{K}I_{j}'(E)$ for every $j$.  
Fix number $j$. For this fixed $j$ denote by
$D_{k}(E)$ be the set of all products of at most  $k$ elements from the set $d_{1}, \ldots , d_{n}$ and an  element from the set $I_{j}(A)$ (in any order).
 Then since $I_{j}'(E)$ is finite-dimensional as a vector space over $K$, we get that for some $f$,
\[D_{f+1}(E)\subseteq  span _{K}D_{f}(E).\]

Therefore, there is  a nonzero polynomial $h(t)\in {\bar {\mathbb C}}\{t\}$ such that 
\[h(t)D_{f+1}(E)\subseteq  span _{K}D_{f}(E).\]

{\em Observation $1$.}  Again, for sufficiently small real numbers $s>0$, we have $h(s)\neq 0$ and all power series from ${\bar {\mathbb C}}\{t\}$ are absolutely convergent at $t=s$.

{\em Observation $2$.} Note that by substituting $t=s$ in products which for elements of $D_{f}(E)$ we are replacing multiplication $\circ _{t}$
 by multiplication $\circ _{s}$ so we obtain elements from $D_{f}(A)$, for any natural number $f$.  
 Moreover, by assumptions, we can choose $s$ which is sufficiently small and $N_{s}$ is isomorphic to $A$, where $N_{s}$ is the specialisation  of $\mathcal N$ at $t=s$. 
By combining observations $1$ and $2$ we obtain:

\[h(s)D_{f+1}(N_{s})\subseteq  span _{{\mathbb C}}D_{f}(N_{s}).\]
 and since $h(s)\neq 0$ we obtain for an appropriate $s$:
\[D_{f+1}(N_{s})\subseteq  span _{{\mathbb C}}D_{f}(N_{s}).\]
 Notice, that by the construction of $D_{f+2}(N_{s})$ we then have 
$D_{f+2}(N_{s})\subseteq span _{\mathbb C}\sum_{i=1}^{n}d_{i}\cdot D_{f+1}(N_{s})+ D_{f+1}(N_{s})\cdot d_{i}\subseteq 
 span _{\mathbb C}\sum_{i=1}^{n}d_{i}\cdot D_{f}(N_{s})+ D_{f}(N_{s})\cdot d_{i}\subseteq {\mathbb C}D_{f+1}(N_{s})\subseteq 
{\mathbb C}D_{f+1}(N_{s})$ (where $\cdot $ denotes the multiplication in $N_{s}$).
 Note that the last inclusion follow from the fact that $D_{f+1}(N_{s})\subseteq  span _{\mathbb C}D_{f}(N_{s})$.

Continuing in this way we obtain for every natural number $i$:
\[D_{f+i}(N_{s})\subseteq span _{\mathbb C} D_{f}(N_{s}).\]

Therefore, $I'_{j}(N_{s})=span _{\mathbb C} D_{f}(N_{s})$.

Analogous argument implies that $I'_{j}(E)=span _{ K} D_{f}(E)$. 

We need to show that $span _{ \mathbb C} D_{f}(A)\leq span _{\mathbb K} D_{f}(E).$
 Because we can take $s$ such that $N_{s}$ is isomorphic to $A$, it suffices to show that 
$span _{ \mathbb C} D_{f}(N_{s})\leq span _{\mathbb K} D_{f}(E).$

Let $span _{\mathbb K} D_{f}(E)$ have dimension $\alpha $ as a vector space over field $K$.
Then there are elements $c_{1}, \ldots , c_{\alpha }\in D_{f}(E)$ such that if $c(t)\in D_{f}(E)$ then 
$c(t)\in Kc_{1}+\ldots +Kc_{\alpha}$.
 Therefore, for each $c(t)\in D_{f}(E)$ there is a nonzero $h(t)\in {\bar {\mathbb C}}\{t\}$ and such that 
\[h(t)c(t)\in {\bar {\mathbb C}}\{t\}c_{1}+\ldots {\bar {\mathbb C}}\{t\}c_{\alpha }.\]
 Because we have finite number of elements in $D_{f}(E)$ then we can have a common $h(t)$  for all  $c(t)\in D_{f}(E)$.
 
Then for sufficiently small $s'>0$ we have $h(s')\neq 0$. 
By specialising at $t=s'$ elements from $D_{f}(E)$ we can obtain any element from $D_{f}(N_{{s'}})$ (by the definition of these sets), as
 we are just replacing multiplication $\circ _{t}$ by the multiplication $\circ _{s'}$ in products.
Recall that $s'$ is such that $h(s')\neq 0$. Therefore,
\[h(t)c(t)\in Kc_{1}'+\ldots Kc_{\alpha }'.\]
 where $c_{1}', \ldots , c_{\alpha }'$ are elements in $D_{f}(N_{s'})$ obtained by specialising elements $c_{1}, \ldots , c_{\alpha}$ at $t=s'$. 
Therefore, the dimension (as a vector space over ${\mathbb C}$) of $D_{f}(N_{s'})$ is at most $\alpha $, so 
$dim_{\mathbb C}D_{f}(N_{s'})\leq \alpha = dim_{\mathbb K}D_{f}(E)$, as required, so
 $dim_{\mathbb C}I_{j}(N_{s'})\leq \alpha = dim_{\mathbb K}I_{j}(E)$. 
 Combining this result with the last two lines of the part $1$ of this proof we obtain that $e_{j}(N_{s'})=e_{j}(E)$ for each $j$. By assumptions we can take $s'$ uch that $N_{s'}$ is isomorphic to $A$, therefore  $e_{j}(A)=e_{j}(E)$ for each $j$. 
 This concludes the proof.
\end{proof} 
 Notice that we can obtain another proof of Proposition \ref{polynomialidentity} by using the following proof of Proposition \ref{main}.  Methods used in the proof of Proposition \ref{polynomialidentity} can be used in a different context, for example in  Section \ref{applications}, so we included them. 
\section{Proposition \ref{main}}

In this section we will prove Proposition \ref{main}. 
 We first introduce some notation. 
\begin{definition}\label{g(x)} 
 Now observe that by the previous subsection 
  $E$ is semisimple since $A$ is semisimple.
 Therefore, $E$ as a $K$-algebra  is isomorphic  to some $K$-algebra
\[ M=\oplus _{i=1}^{n'} M_{n_{i}}(D_{i})  ,\]
 where $D_{i}$ is a finite-dimensional skew-field extension of $K$. 

 By Remark \ref{commutative}  $D_{i}$ are commutative, hence we can find polynomials $g_{i}(x)$ with entries in $K$ and elements $u_{i}\in D_{i}$ such that $D_{i}=K[u_{i}]$, and $g_{i}(u_{i})=0$. Notice that $g_{i}(x)$ is the minimal polynomial for $u_{i}$. 

 Moreover we can assume that entries of $g_{i}(x)$ are from ${\bar C}\{t\}$, and that the entry at the largest power of $x$ is $1$ (as if it is not the case then instead of $u_{i}$ we can take element $u_{i}h(t)^{-1}$ for some $h(t)\in {\bar C}\{t\}$).
  Since $D_{i}$ is a division ring, it follows that $g_{i}(x)\in K[x]$ is an irreducible polynomial (and hence has no multiple roots, since $K$ has characteristic zero).  

Let $s>0$ be a small  real number, 
 by $g_{i,s}(x)$ we denote the polynomial obtained by specialising entries of $g_{i,s}(x)$ at $t=s$.  
 Then,  $g_{i,s}(x)$ is the minimal polynomial for $v_{i}$ where $v_{i}\in {\mathbb C}[y]/ \langle g_{i,s}(y)\rangle$ where $<g_{i,s}(y)>$ is the ideal generated by the polynomial $g_{i,s}(y)$ in the polynomial ring  ${\mathbb C}[y]$ and where $v_{i}= y+\langle g_{i,s}(y)\rangle$.   
 Notice that for sufficiently small $s$, $\deg g_{i}(x)=\deg g_{i,s}(x)$.  Moreover, we assume $v_{i}v_{j}=0$ for $i\neq j$ (as they may appear as parts of different simple components of a semisimple algebra, and then it is convenient to define their multiplication and define it  as zero, as the multiplication of elements form different simple components in a semisimple ring are zero). 
\end{definition}
\begin{definition}\label {sigma}
 Let the notation be as in Definition \ref{g(x)}, so $u_{i}, v_{i}$ are as in Definition \ref{g(x)}. Let ${\mathbb C}_{s}\{t\}$ be the subring of the power series ring ${\mathbb C}\{t\}$ consisting of power series which are convergent for $t\in [0, s+e)$ for some real number $e>0$. 
  
 The homomorphism of  $\mathbb C$-algebras  $\sigma :\oplus _{i=1}^{n'}{\mathbb C}_{s}\{t\}[u_{i}]\rightarrow \oplus _{i=1}^{n'}{\mathbb C}[v_{i}]$ is defined by: $\sigma (u_{i}^{j}) =v_{i}^{j}$ 
 and extended to the homomorphism of $\mathbb C$-algebras (note that it is well defined as $\sigma (g_{i}(x))=g_{i,s}(x)$ where $g_{i,s}(v_{i})=0$ and $ g_{i}(u_{i})=0$).

\end{definition}

\begin{notation}\label{m1} 
 Let $N, A$ be finite-dimensional $\mathbb C$-algebras, of the same dimension $n$, with $A$ semisimple. Let $\mathcal N$ be a strongly flat deformation from $N$ to $A$. Let notation be as in Definition \ref{76}. 
 Let $s>0$ be a real number.

  Let \[d_{1}, d_{2}, \ldots , d_{n}\] span the algebra $N$ as a vector space. 
 Let \[d_{1}', d_{2}', \ldots , d_{n}'\] be the corresponding elements of ${\bar {\mathcal N}}$ (and of $\mathcal N$ as  ${\bar {\mathcal N}}\subseteq \bar {\mathcal N}$).
  Let \[d_{1}'', d_{2}'', \ldots , d_{n}''\] be the corresponding elements of $E$,
  and let \[d_{1,s}, \ldots , d_{n, s}\] be the corresponding elements of $N_{s}$ (the specialisation of $\mathcal N$ at $t=s$).
\end{notation}
 We start with the following remark
\begin{remark}\label{xy}
 Let notation be as in Notation \ref{m1}. 
  Consider the following relations which hold in ${\bar {\mathcal N}}$: 
\[d_{j}'\circ _{t} d_{j}'=\sum_{i=1}^{n}g_{i,j,k}(d_{j}',d_{k}')t^{i}\] 
where $u_{i,j,k}(d_{j}', d_{k}')\in \mathbb C$.
 Then these relations form the Groebner basis of the ${\mathbb C}\{t\}$-algebra $\mathcal N$, so by  applying these relations we get  \[(d_{i}'\circ d_{j}')\circ d_{k}'=\sum_{i=1}^{n}\alpha _{i,j,k,l}(t)d_{l}'\] and 
\[d_{i}'\circ (d_{j}'\circ d_{k}')=\sum_{i=1}^{n}\beta _{i,j,k,l}(t)d_{l}'\]  with $\alpha _{i,j,k}(t)=\beta _{i,j,k}(t)$, where $\alpha _{i,j,k,l}(t), \beta _{i,j,k,l}(t)\in {\bar {\mathbb C}}\{t\}$.  

Indeed, if the above relations didn't form Groebner bases, then 
$(d_{i}'\circ d_{j}')\circ d_{k}'-d_{i}'\circ (d_{j}'\circ d_{k}')$ would be a nontrivial linear combination of elements $d_{l}'$, which would force elements $d_{l}'$ be linearly dependent over ${\bar {\mathbb C}}\{t\}$, a contradiction
 (since by the construction of the formal deformation $\mathcal N$ it is spanned by elements $d_{1}', \ldots , d_{n}'$).

 
 Observe also that the  specialisation of the above relations  at $t=s$  gives a Groebner basis for  $N_{s}$.
 Indeed, specialising at $t=s$ we get  relations \[d_{j,s}\circ _{s} d_{k,s}=\sum_{i=1}^{n}u_{i,j,k}(d_{k,s},d_{j,s})s^{i}.\]
  Then,  $\alpha _{i,j,k}(t)-\beta _{i,j,k}(t)=0$ implies that $\alpha _{i,j,k}(s)-\beta _{i,j,k}(s)=0$  so the above relations form a Groebner basis for $N_{s}$.  
\end{remark}

{\bf Proof of Proposition  \ref{main}.}  Let notation be as in Notation \ref{main}. 
 In the algebra $\bar {\mathcal N}$ we have relations  
\[d_{j}'\circ_{t} d_{k}'=\sum_{i=1}^{n}\alpha _{i,j,k}(t)d_{i}',\] for some $\alpha _{i,j,k}\in {\bar { \mathbb C}}\{t\}$ (where $\circ _{t}$ is the multiplication in $\mathcal N$ and in ${\bar {\mathcal N}}$). 
 Notice that the  power series $\alpha _{i,j,k}(t)$ are convergent for sufficiently small $s$, because the deformation $\mathcal N$ is strongly flat. 
 Consequently, the corresponding relations also hold in $E$. 
 So we have: 
\[d_{j}''\circ _{t} d_{k}''=\sum_{i=1}^{n}\alpha _{i,j,k}(t)d_{i}'',\]
 where $\circ _{t}$ denotes the multiplication in $E$.
 Each $d_{i}''$ is an element of $E$, and hence can be written as $h(t)^{-1}{\bar d}_{i}$ where $h(t)\in {\bar {\mathbb C}}_{s}\{t\}$ and 
  entries of ${\bar d}_{i}\in E$ are from \[\{{\bar {\mathbb C}_{s}}\{t\}u_{j}^{l}: j,l\geq 0\}\] and entries of polynomials $g_{i}(x)$ are from ${\bar {\mathbb C}}\{t\}_{s}u_{j}^{t}$ (it holds for all sufficiently small $s$, see Definition \ref{sigma}).

Let $s>0$ be a real number such that all entries of  matrices $d_{1}'', d_{2}'', \ldots , d_{n}''$ from $E$ are well defined and $\alpha _{i,j,k}$ are convergent at $(0,s+\epsilon)$ for some real number $\epsilon >0$. 
Let \[d_{j}'''\in \oplus _{i=1}^{n'}M_{n_{i}}({\mathbb C}[v_{i}])\] be obtained by applying  homomorphism $\sigma $ from Definition \ref{sigma} to each entry of  matrices $d_{i}''$.
Let $N_{s}'$ be the $\mathbb C$-algebra spanned by  matrices $d_{1}''', d_{2}''', \ldots , d_{n}'''$. 

Then they satisfy the relations:
\[d_{j}'''\circ  d_{k}'''=\sum_{i=1}^{n}\alpha _{i,j,k}(s)d_{i}''',\]
 since the relations $d_{j}'\circ _{t} d_{k}'=\sum_{i=1}^{n}\alpha _{i,j,k}(t)d_{i}'$ hold in  $\mathcal N$ and $\sigma (\alpha _{i,j,k}(t))=\alpha _{i,j,k}(s)$, and $\sigma $ is a homomorphism of $\mathbb C$-algebras.  

Observe on the other hand that since $\mathcal N$ satisfies relations:
\[d_{j}'\circ _{t} d_{k}'=\sum_{i=1}^{n}\alpha _{i,j,k}(t)d_{i}'\]
 then in $N_{s}$, for a sufficiently small $s$, we have relations:
\[d_{j,s}\circ_{s} d_{k,s}=\sum_{i=1}^{n}\alpha _{i,j,k}(s)d_{i,s}.\]
 By Remark \ref{xy} these relations for a Groebner base of $N_{s}$, so they are  defining relations for $N_{s}$. 

 Note that $N_{s}'$ is spanned by $n$ elements $d_{1}''', \ldots , d_{n}'''$ so is at most $n$-dimensional.
 Also, since $N_{s}'$ satisfies the above defining relations of $N_{s}$ (where $d_{i,s}$ corresponds to $d_{i}'''$), therefore the $\mathbb C$-algebra $N_{s}'$ is a homomorphic image of the $\mathbb C$-algebra $N_{s}$. 
Consequently, we see that  $N_{s}$ is isomorphic to $N_{s}'$, provided that $N_{s}'$ has dimension $n$.
 
{\em Part $1$.}
 We will now show that $N_{s}'$ has dimension $n$ for sufficiently small real numbers $s>0$. 
  Notice that each $d_{i}''$ is in \[C=t^{-\gamma}\oplus _{i=1}^{n'}M_{n_{i}}({\bar {\mathbb C}}\{t\}+\cdots +{\bar {\mathbb C}}\{t\}u_{i}^{ \deg (g_{i}(x))-1})\subseteq E,\] for some integer $\gamma $ (since $d_{1}, \ldots , d_{n}$is a finite set and $\mathcal N $ is a formal deformation so $d_{i}\circ _{t} d_{j}\in \sum_{i=1}^{n}{\mathbb C}\{t\}d_{i}$).
  Let $e_{1}, \ldots , e_{n}$ be a basis of this linear space which is made by matrices whose all entries are zero except of one entry which is either $1$ or $u_{i}^{t}$ for some $0<t<\deg (g_{i}(x))$.  
  
Let $e_{1}', \ldots , e_{n}'$ be the corresponding base of the algebra 
\[B=\oplus _{i=1}^{n'}M_{n_{i}}({ {\mathbb C}}+\cdots +{ {\mathbb C}}v_{i}^{ \deg (g_{i}(x))-1}),\]
 so each $e_{i}'$ is a matrix whose all entries are zero except of one entry which is either $1$ or $v_{i}^{t}$ for some $t<\deg (g_{i}(x))$.  Notice that $\sigma (e_{i})=e_{i}'$.   

We can write each $d_{i}''$ as a linear combination of elements  $e_{1}, \ldots , e_{n}$  and with entries from $K$.
 We obtain that for some  $h(t)$ and each $\alpha _{i,j}(t)\in {\bar {\mathbb C}}\{t\}$ we have 
\[ h(t)d_{j}''=\sum_{i=1}^{n}\alpha _{i,j}(t)e_{i} .\]
 Now we can take a sufficiently small $s$ such that entries of matrices $d_{i}''$ are convergent for $t\in (0, s+e)$ for some real number $e>0$. 
 By taking homomorphism $\sigma $ to each entry of matrices $d_{i}$ (for sufficiently small $s$ so $h(s)$ is not zero) we obtain 
\[h(s) d_{j}'''=\sum_{i=1}^{n}\alpha _{i,j}(s)e_{i}'.\]
Notice that since $d_{1}'', \ldots , d_{n}''$ are linearly independent over $K$ it follows that the determinant of the  $n$ by $n$ matrix  whose $i,j$-entry is $\alpha _{i,j}(t)$ is not zero, so it is a power series from ${\bar {\mathbb C}}\{t
\}$. For a sufficiently small $s>0$ this power series specialised at $t=s$ is not zero (since it is convergent so coefficient at $t^{i}$ is smaller than $c^{i}$ for some constant $c$ and for each $i$). Therefore, 
the determinant of the $n$ by $n$ matrix whose $i,j$-entry is $\alpha _{i,j}(s)$ is nonzero.
 Consequently, elements $d_{1}''', \ldots , d_{n}'''$ are linearly independent over ${\mathbb C}$ as required (for sufficiently small $s$).

{\em Part $2$.} We will now show that 
$N_{s}'$ is isomorphic to $A$, for sufficiently small $s$. We will first show that the polynomial $g_{i,s}(x)$ obtained by specialising entries of the polynomial $g_{i}(x)$ at $t=s$ has no multiple roots (for sufficiently small $s$).
 This follows, because if $F$ is a field of characteristic zero then a polynomial $g(x)\in F[x]$ has multiple roots, if  and only if the polynomials $g(x)$ and its derivative $g'(x)$ have a common root, and hence a common factor. The proof is then the same as the proof of Lemma $8$, \cite{Wemyss1}. 
 
We now see that by the construction that $N_{s}'$ is a $\mathbb C$-subalgebra of the algebra
\[\oplus _{i=1}^{n'} M_{n_{i}}({ {\mathbb C}}[v_{i}]).\]  We know by assumptions that this $\mathbb C$-algebra has dimension $n$ (since $E$ has dimension $n$ as $K$-algebra by Remark \ref{xy}).  Recall that  $N_{s}'$ is $n$-dimensional algebra by part $1$ above. It follows that 
$N_{s}'= \oplus _{i=1}^{n'} M_{n_{i}}({{\mathbb C}}[v_{i}])$.
 It suffices to show that $\oplus _{i=1}^{n'} M_{n_{i}}({ {\mathbb C}}[v_{i}])$ is isomorphic to $A$ as $\mathbb C$-algebra. This follows since $g_{i,s}(x)$ has no multiple roots and hence ${\mathbb C}[v_{i}]$,  is a semisimple algebra, isomorphic to ${\mathbb C}^{\oplus \deg (g_{i}(x))}$. 
 Therefore, $N_{s}'$ is a semisimple algebra for sufficiently small $s'$.   

$ $

{\em Part $3$.} Recall that  $\mathcal N$ is a strongly flat deformation from $N$ to $A$. Therefore, by Proposition \ref{polynomialidentity} we have
$e_{j}(E)=e_{j}(A)$. Observe on the other hand,  that 
  by part $2$, for each $j$ and for sufficiently small $s$, $e_{j}(N_{s})=e_{j}(E)$. Hence for sufficiently small $s$, 
 $N_{s}$ is isomorphic to $A$.

\begin{proposition}\label{Aga}
 Let $N, A'$ be two finite-dimensional $\mathbb C$-algebras of the same dimension $n$ with $A'$ semisimple.   Let $\mathcal N$ be a flat deformation of $N$ (in the sense of Definition \ref{flat}) which deforms to $A'$ at $t=s$.  Let ${\bar {\mathcal N}}$ be the ${\bar {\mathbb C}}$-algebra defined in  Definition \ref{76}. 
  Suppose that  ${\bar {\mathcal N}}$ is a semisimple ${\bar {\mathbb C}}$-algebra. Let $E$ be as in Definition \ref{76}. Then  $E$ is semisimple. Moreover, $\mathcal N$ is a strongly flat deformation from $N$ to some semisimple $\mathbb C$-algebra $A$. \end{proposition}
\begin{proof}  {\em Part $1$.} We will first show that $E$ is well defined and semisimple. Note that ${\bar {\mathcal N}}$ is a ${\bar {\mathbb C}}$-algebra, and hence it can be embedded into matrix algebra 
 $M_{n}({\bar {\mathbb C}})\subseteq M_{n}(K)$, so $E$  is well defined.
 
 Now we will show that $E$ is a  semisimple algebra. Suppose on the contrary that  $E$ has a nilpotent ideal generated by some element $d_{1}h_{1}(t)+\ldots +d_{n}h_{n}(t)$ for  
$h_{1}, \ldots , h_{n}\in K$, then by multiplying $h_{1}, \ldots , h_{n}$ by appropriate $h(t)\in {\bar {\mathbb C}}$ we obtain that $\mathcal N$ has a nilpotent ideal generated by $h(t)h_{1}(t)d_{1}+\ldots +h(t)h_{n}(t)d_{n}(t)$.
 A contradiction as ${\bar {\mathcal N}}$ is semisimple.  

{\em Part $2$.} The same proof as in Proposition \ref{main} works, when taking only  parts $1$ and $ 2$ of this proof.  Therefore,   for sufficiently small $s>0$,
 $\mathcal N$ specialised at $t=s$ gives an algebra isomorphic to some semisimple algebra $A$. 
 Because there is only a finite number of non-isomorphic semisimple $\mathbb C$-algebras of dimension $n$, it  follows that $\mathcal N$ is a strongly flat deformation to some algebra $A$.

\end{proof}

\section{Applications}\label{applications}

 We first introduce some notation.
\begin{notation}\label{et}
Let $N, A$ be finite-dimensional $\mathbb C$-algebras with $A$ semisimple, and let $\mathcal N$ be a strongly flat deformation from $N$ to $A$. 
 Let $s>0$ be a real number (sufficiently small so $N_{s}$ is defined for this $s$).

Let $d_{i}, d_{i}', d_{i}'', d_{i,s}$ be as in Notation \ref{m1} for $i=1, \ldots , n$.

 Let $w_{i}(d_{1}, \ldots , d_{n})$ be a linear combination over $\mathbb C$ of products of elements $d_{1}, \ldots , d_{n}\in N$ for $i=1,2, \ldots  $ under the multiplication from $N$.
  
Let $T(d_{1}, \ldots , d_{n})$ be the linear space over $\mathbb C$ spanned by elements $w_{i}(d_{1}, \ldots , d_{k})$ for $i=1,2, \ldots  $. 

Let $w_{i,s}(d_{1,s}, \ldots , d_{n,s})$ be the  corresponding linear combination over $\mathbb C$ of products of elements $d_{1,s}, \ldots , d_{n,s}\in N_{s}$ for $i=1,2, \ldots , $ under the multiplication from $A=N_{s}$.

 That means that in $w_{i,s}(d_{1,s}, \ldots , d_{n,s})$ we used element $d_{j,s}$ instead of $d_{j}$ and we used multiplication from $N_{s}$ instead of multiplication from $ N$ to multiply elements $d_{j,s}$.

 Let $T_{s}(d_{1,s}, \ldots , d_{n,s})$ be the linear space over $\mathbb C$ spanned by elements $w_{i,s}(d_{1,s}, \ldots , d_{n,s})$ for $i=1,2, \ldots  $.

 Let $w_{i,{\mathcal N}}(d_{1}', \ldots , d_{n}')$ be the  corresponding linear combination over $\mathbb C$ of products of elements $d_{1}', \ldots , d_{n}'\in {\bar {\mathcal N}}$ for $i=1,2, \ldots , n$ under the multiplication from ${\bar {\mathcal N}}$.
 Let $T_{\mathcal N}(d_{1}', \ldots , d_{n})'$ be the linear space over $\mathbb C$ spanned by elements $w_{i}(d_{1}, \ldots , d_{n})$ for $i=1,2, \ldots $.

Let 
$w_{i,E}(d_{1}'', \ldots , d_{n})$ be the  corresponding linear combination over $\mathbb C$ of products of elements $d_{1}'', \ldots , d_{n}''\in E$ for $i=1,2, \ldots $ under the multiplication from $E$.

Let $T_{E}(d_{1}'', \ldots , d_{n}'')$ be the linear $K$-space of $E$ over $\mathbb C$ spanned by elements $w_{i}(d_{1}'', \ldots , d_{n}'')$ for $i=1,2, \ldots $.

\end{notation}
 We have the following application of the Theorem \ref{polynomialidentity}. 
\begin{proposition}\label{ce}  Let notation be as in Notation \ref{et}.  
Then, for every real number $e>0$ there exists $s<e$ and elements $b_{1}, \ldots , b_{n}\in N_{s}$ such that 
\[\dim_{\mathbb C}T(d_{1}, \ldots , d_{n})\leq \dim_{\mathbb C}T_{s}(b_{1}, \ldots , b_{n}).\]
 
\end{proposition}
\begin{proof}
  {\em Part $1$.} Observe first that 
 \[\dim_{K} T_{E}(d_{1}'', \ldots , d_{n}'')\geq \dim_{\mathbb C}T(d_{1}, \ldots , d_{n}).\]
 To prove the above inequality we can  consider only these $w_{i}(d_{1}, \ldots , d_{n})$ which are nonzero and linearly independent over $\mathbb C$. Without restricting the generality we can assume that these are $w_{i}(d_{1}, \ldots , d_{n})$ for 
 $i\leq m'$ for some $m'$.
 Then the corresponding $w_{i,{\mathcal N}}(d_{1}', \ldots , d_{n}')$
 are linearly independent over $\mathbb C$, as otherwise there would be (not all zero) $\alpha _{i}(t)\in {\mathbb C}\{t\}$ such that 
 $\sum_{i=1}^{m'}\alpha _{i}w_{i, {\mathcal N}}(d_{1}', \ldots , d_{n}')=0$, by comparing the elements at the lowest possible power of $t$ (which has a non-zero component) we obtain that 
 it is not possible (since $w_{i}(d_{1}, \ldots , d_{n})$ which are nonzero and linearly independent over $\mathbb C$).

 This implies that elements  $w_{i, {\mathcal N}}(d_{1}', \ldots , d_{n}')$ are linearly independent over $K$.
 This concludes the proof that 
 \[\dim_{K} T_{E}(d_{1}'', \ldots , d_{k}'')\geq \dim _{\mathbb C}T(d_{1}, \ldots , d_{n}).\]

{\em Part $2$.} It suffices to show that \[\dim_{K} T_{E}(d_{1}'', \ldots , d_{n}'')\leq \dim_{\mathbb C}T_{s}(d_{1,s}, \ldots , d_{n,s}).\] 
The proof is very similar to the part $1$ of the proof of Proposition \ref{polynomialidentity}.
 
 Suppose that $\dim_{K} E-\dim_{K} T_{E}(d_{1}'', \ldots , d_{n}'') = \beta $.

As the enumeration of elements $d_{i}''$ does not matter, we can assume that 
 \[E=span _{K}\sum_{i=1}^{\beta }d_{i''}K+T_{E}(d_{1}'', \ldots , d_{n}'').\]
Therefore, for some $ g_{1}(t), \ldots , g_{\beta }(t)\in K$ we have 
\[d_{\beta +1}''\in \sum_{i=1}^{\beta } g_{i}(t)d_{i}''+T_{E}(d_{1}'', \ldots , d_{n}'').\]

Therefore there are $h(t), h_{1}(t), \ldots , h_{\beta }(t)\in {\bar {\mathbb C}}\{t\}$
 with $h(t)\neq 0$ and $c(t)\in   {\bar {\mathbb C}}\{t\}T_{E}(d_{1}'', \ldots , d_{n}'')$ 
 such that 
\[h(t)d_{\beta +1}''= \sum_{i=1}^{\beta } h_{i}(t)d_{i}''+c(t).\]

 We can assume that $c({t})$ is a linear combination over  ${\bar {\mathbb C}}\{t\}$ of elements from the set $\{d_{1}'', \ldots , d_{n}''\}$ (because if needed we can adjust $h(t)$).
 So we have \[c(t)=h(t)d_{\beta +1}-\sum_{i=1}^{\beta } h_{i}(t)d_{i}''.\]

This is  relation which holds in the ${\bar {\mathbb C}}\{t\}$-algebra $\bar {\mathcal N}$. 
 Recall that $\bar {\mathcal N}$ specialised at $t=s$ gives multiplication in $N_{s}$-the algebra spanned by $d_{,s}, \ldots , d_{n,s}$ (as in Notation \ref{et}).

Let $0<s'<s$.  By specialising $c(t)$ in $t=s'$ we get relation in $N_{s'}$. 
\[c(s')=h(s')d_{\beta +1}-\sum_{i=1}^{\beta } h_{i}(s')d_{i},\]
 where $c(s')$ is obtained from $c(t)$ by specialising at each place
 the multiplication for $t=s'$, for example if $c(t)=(d_{1}\circ _{t}d_{2}-d_{2}\circ _{t}d_{1})\circ _{t}d_{3}$, 
then $c(s')=(d_{1}\circ _{s'}d_{2}-d_{2}\circ _{s'}d_{1})\circ _{s'}d_{3}$,
 Note that $c(s')\in T(d_{1,s}, \ldots , d_{k,s})$ since $c(t)\in T(d_{1}'', \ldots , d_{k}'')$, by the construction. 

$ $

Because $h_{i}(t), h(t)$ are absolutely convergent for $0< t<s$  
 there exists $e>0$ such that for  all $0<s'<e$ we have  $h(s')\neq 0$, therefore by specialising at $t=s'$ we get,
 
\[d_{\beta +1,s'}\in \sum_{i=1}^{\beta }{ \mathbb C}d_{i,s'}+T(d_{1,s'}, \ldots , d_{n,s'})\]

We can apply the same argument and obtain for $i>0$  
\[d_{\beta +i,s'}\in \sum_{i=1}^{\beta }{ \mathbb C}d_{i,s'}+T(d_{1,s'}, \ldots , d_{n,s'})\]

Therefore $dim_{\mathbb C}T(d_{1,s'}, \ldots , d_{n,s'})\geq n-\beta $. Since 
 by assumption $dim_{K}T_{E}(d_{1}'', \ldots , d_{n}'')=n-\beta $ we obtain 
$dim_{\mathbb C}T(d_{1,s'}, \ldots , d_{n,s'})\geq dim_{K}T_{E}(d_{1}'', \ldots , d_{n}'')$ for every $j$.
\end{proof}

Proposition \ref{ce} gives a partial answer to question $6.5$ from \cite{DDS}.
 
\begin{corollary}\label{cde} 
 Let $N$ be a $\mathbb C$-algebra which is generated by elements $x,y$, and suppose that elements $1, x, \cdots x^{n}$ and $y, yx, \cdots , yx^{n}$ span $N$ as a linear space over $\mathbb C$. 
 Then $N$ cannot be deformed strongly flatly to any semisimple algebra $A$ (of the same dimension) which contains as a summand the matrix algebra of dimension larger than $2$ (over complex numbers).
\end{corollary}
\begin{proof} We use a proof by contradiction.  Let $n=\dim_{\mathbb C}A=\dim _{\mathbb C}N$. Take $w_{2i}(x,y)=x^{i}$, $w_{2i+1}(x,y)=yx^{i}$ for $i=0, 1,  \ldots, n$. 
 Denote $d_{1}=x, d_{2}=y$ then $T(d_{1}, d_{2})=T(x,y)$ has dimension $n$.

Let $b_{1}, b_{2}\in N_{s}$. Note that $T_{s}(b_{1}, b_{2})$ 
 cannot have dimension $n$, as then it would span the summand (which we call $S$)  of $N_{s}=A$ which is a matrix ring of dimension $j>2$, so it has dimension at least $j^{2}\geq 3^{2}$. Let $b_{1}', b_{2}'$ be the summands of $b_{1}, b_{2}$ corresponding to $S$. Then elements $b_{1}'^{i}$ and $b_{2}'b_{1}'^{i}$ span $S$ (where multiplication is in this matrix ring), so they span vector space of dimension $j^{2}$.
However elements  $1=b_{1}'^{0}, b_{1}'^{1}, b_{1}'^{2},\cdots , b_{1}'^{j}$ (in the summand corresponding to this matrix ring) are linearly dependent over $\mathbb C$ 
  for $i\geq j$ (because every matrix of dimension $j$ satisfies its characteristic polynomial which has degree $j$) so we have at most 2j elements linearly independent over $\mathbb C$ in the vector space spanned by elements
 $b_{1}'^{i}$ and $b_{2}'b_{1}'^{i}$ for $i=0,1, \ldots , n$.
 A contradiction.
\end{proof}

Proposition \ref{ce}  can be  used to investigate a question of Michael Wemyss  find semisimple algebras to which a given contraction algebra cannot  be deformed.

\begin{example}\label{ef}  There are many Acons  in the paper \cite{DW} which satisfy the  assumptions of Corollary \ref{cde}, since they are generated by elements $x,y$ and satisfy relation $yx=-xy$, hence elements $x^{i}$, $yx^{i}$ span 
this Acon as a linear space over $\mathbb C$. For example the following Acon of dimension $12$ generated by $x,y$ and satisfying the following defining relations:
$y^{6}-x^{3}-y^{2}x$, $y^{4}x+x^{2}+y^{2}$, $x^{4}-y^{4}$, $yx^{2}+y^{3}$, $xy+yx$.

 Corollary \ref{cde} shows that the only semisimple algebra $A$ to which  this Acon can be deformed strongly flatly  are direct sums of $\mathbb C$ and $M_{2}({\mathbb C})$ (the two by two matrices with entries from $\mathbb C$). 
\end{example}

\section{Results related to the comment from Joachim Jelisiejew}
 We now provide the proof of the case (ii) from the introduction  cannot hold, using  the idea suggested by Joachim Jelisiejew  as in the introduction:
\begin{proposition}\label{JJ}
 Let $N, A$ be two finite-dimensional $\mathbb C$-algebras of the same dimension $n$ with $A$ semisimple.   If $\mathcal N$ is a flat deformation of $N$ (in the sense of Definition \ref{flat}) which deforms to $A$ at $t=s$ 
 then $\mathcal N$ deforms to $A$ at $t\in [s-\epsilon, s+\epsilon ]$ for some real number $\epsilon >0$. Moreover,   the algebra  ${\bar {\mathcal N}}$ defined in Definition \ref{76}  is a semisimple ${\bar {\mathbb C}}$-algebra.  
\end{proposition}
\begin{proof} By Proposition \ref{Aga} it suffices to consider the case when ${\bar {\mathcal N}}$ 
 has a non-zero nilpotent ideal. By Proposition $2.5$ \cite{Gabriel}   there is a real number $\epsilon $ such that $\mathcal N$ at $t\in [s-\epsilon, s+\epsilon ]$ deforms to $A$ (as explained in the introduction in the comment from Joachim Jelisiejew). 

 Suppose now on the contrary that ${\bar {\mathcal N}}$ is not semisimple. Then it has a non-zero nilpotent ideal, generated  by some power series $\sum d_{i}h_{i}(t)$ where $d_{i}$ are as in Definition \ref{76} and $h(i)\in {\bar {\mathbb C}}\{t\}$.
 Moreover, we can assume that $h_{i}(t)$ are convergent at $(0,c)$ because the formulas for multiplication in $\mathcal N$ use power series convergent at some interval $(0,c)$, where $c$ is as in Definition \ref{stronglyflat}, so $s\in (0,c)$ (see additional explanations at the end of this proof). 
 
 By a theorem from analytic series there exists $q\in [s-\epsilon, s+\epsilon ]$ such that $\sum d_{i}h_{i}(q)$ is a non-zero element in the $N_{q}$ (the specialisation of $\mathcal N$ at $t=q$). Therefore, $\sum d_{i}h_{i}(q)$ generates a non-zero nilpotent ideal in $N_{q}$.  Recall that  $N_{q}$ is isomorphic to $A$, and hence is semisimple, a contradiction.  

{\em Explanation why we can choose $h_{i}(x)$ which are convergent at $[0,c)$}:
 Let ${{\mathbb C}\{t\}}_{c}$ denote the set of power series convergent at $t\in [0, c)$. Then, let  ${\mathcal N} ^{c}$ denote the algebra over ${{\mathbb C}\{t\}}_{c}$ defined analogously as ${\bar {\mathcal N}}$ but using 
${{\mathbb C}\{t\}}_{c}$ instead of ${\bar {\mathbb C}}\{t\}$. So 
 ${\mathcal N}^{c}$ is an algebra  over the algebra of power series convergent at some interval $(0,c)$ (where $c$ is as in Definition \ref{stronglyflat}, so $s\in (0,c)$).  If ${\mathcal N} ^{c}$ is semisimple then it can be embedded into some algebra $E^{c}$ (constructed analogously as $E$) which is a direct sum of matrix rings over $K_{c}$ where $K_{c}$ is the field of fractions of ${{\mathbb C}\{t\}}_{c}$, and it has dimension $n$. Then by extending scalars 
 we get that ${\bar {\mathcal N}}$ can be embedded into $E_{c}\otimes _{K_{c}}K$ which a direct sum of matrix rings over $K$, which also has dimension $n$ as a vector space over $K$. It follows that $\mathcal {\bar N}$ is semisimple, as it cannot have nilpotent ideals. 
\end{proof} 

From Propositions \ref{JJ} and \ref{Aga} we obtain: 
\begin{corollary}\label{west}
 Let $N, B$ be two finite-dimensional $\mathbb C$-algebras of the same dimension $n$ with $B$ semisimple and let $\mathcal N$ be a formal deformation of $N$.   Suppose that $\mathcal N$ is a flat deformation of $N$ (in the sense of Definition \ref{flat}) which deforms to $B$ 
 (when specialised at $t=t_{0}$ for some $0<t_{0}\in \mathbb R$).

Then, it is not possible that  for all sufficiently small real numbers $s>0$ the algebra $\mathcal N$ specialised at $t=s$ is not semisimple, and hence it has a non-zero nilpotent ideal. 
\end{corollary}
\begin{proof}  Notice that ${\bar {\mathcal N}}$ specialised at $t=t_{0}$ gives $B$ for some $t_{0}\in \mathbb R$ (in the sense of  Definition \ref{flat}),  and hence
 ${\bar {\mathcal N}}$ is semisimple by Proposition \ref{JJ}.
  Let $E$ be as in Definition \ref{76}. By Proposition \ref{Aga} the algebra $E$ is semisimple.  

  By Proposition \ref{Aga}, since $\mathcal N$ is semisimple. 
  then for sufficiently small $s>0$
 $\mathcal N$ specialised at $t=s$ gives an algebra isomorphic to some semisimple algebra $A$. 
 This concludes the proof. 
\end{proof}


We now introduce a definition:

 \begin{definition}\label{type} (A type of an algebra $R$).
 Let $R$ be a finitely-dimensional semisimple algebra over a commutative algebra $S$ (so $R$ is an $S$-algebra) and let $d_{1}, \ldots , d_{n}$ be a basis of this algebra as a linear space over $S$. Assume that $S$  is a $\mathbb C$-  subalgebra of ${\mathbb C}\{x\}$.  
Let $I_{m}'$ denote the ideal generated in $R$ by elements $s_{2m}(a_{1}, \ldots , a_{2m})$ where $s_{2m}$ is the identity for $m$ by $m$ matrices over commutative rings. Note that  $s_{2m}(a_{1}, \ldots , a_{2m})$ is defined in the proof of Proposition \ref{polynomialidentity} for all  $a_{1}, \ldots , a_{2m}\in R$ (so $I_{m}(a_{1}, \ldots , a_{2m})$ is zero if $a_{1}, \ldots , a_{m}$ belong to a $m$ by $m$ matrix ring with entries from some commutative ring).

 By $J_{m}$ we denote the $S$-linear subspace of $R$ such that $r\in J_{m}$ if and only if $r\cdot s\in I_{m}$ for some $s\in S$, $s\neq 0$. Note that if $r\cdot q\in J_{m}$ for some $0\neq q\in S$ then $r\in J_{m}$. 
 Let $c_{j}$ be the smallest number such that there are elements $r_{1}, \ldots , r_{c_{j}}\in I_{j}'$ such that 
  if $r\in J_{j}$ then $qr\in Sr_{1}+\cdots +Sr_{c_{j}}$ for some $q\neq 0, q\in S$. 
 
  We call the sequence $c_{1}, c_{2},\ldots $ the type of $A$.

Observe that if $S$ is a field, and $R=\oplus _{i=1}^{n'}M_{n_{i}}( D_{i})$ where each $D_{i}$ a finitely dimensional  field extension of $S$ then for each $i$,   $c_{i-1}-c_{i}=dim_{S}e_{i}(R)$ where  $e_{i}(R)$ is defined  
 in Definition \ref{12345} (for $K=S$).  This follows since matrix rings of dimension $m$ over commutative rings satisfy the identity $s_{2m}$, and matrix rings of dimension $m+1$ don't satisfy this identity, moreover matrix rings over fields are simple rings.
\end{definition}
 We will now provide a result obtained by combining the comment of Joachim Jelisiejew from the introduction  with some methods from section \ref{12345}. 
 We will now prove Therem \ref{main2}, which we present in a slightly reformulated form:

\begin{theorem} 
 Let $N, B, C$ be three finite-dimensional $\mathbb C$-algebras of the same dimension $n$ with $B$ and $C$  semisimple and let $\mathcal N$ be a formal deformation of $N$, hence $\mathcal N$ specialised at $t=0$ gives algebra $N.$ Suppose that $\mathcal N$ gives a flat deformation, in the sense of Definition \ref{flat}, from $N$ to $B$, so $\mathcal N$ specialised at $t=t_{0}$   (for some nonnegative  real number $t_{0}$) gives algebra $B$. Suppose, that $\mathcal N$ gives a flat deformation, in the sense of Definition \ref{flat}, from $N$ to $C$, so $\mathcal N$ specialised at  $t=t_{1}$  (for some nonnegative  real number $t_{1}$) gives algebra $C$.
 Then $B$ and $C$ are isomorphic $\mathbb C$-algebras.
\end{theorem}
 \begin{proof}  {\em An outline of the proof.}  By Propositions \ref{Aga} and \ref{JJ} and Corollary \ref{west} $\mathcal N$ is a strongly flat deformation to some semisimple algebra $A$. First it is shown that $A$ and $\mathcal N$  have the same type. 

First we will show that because $\mathcal N$ is a strongly flat deformation to some algebra $A$ then  the $\mathbb C$-algebra $A$ and $\mathbb C$-algebra $\mathcal N$ have the same type (where the type of an algebra is defined in  Definition \ref{type}). In the part $2$ we show that because specialised  at some interval gives algebras isomorphic to $B$ (by Gabriel's theorem as described in the introduction by Joachim Jelisiejew)  it follows that $\mathcal N$ and $B$ have the same type. This follows because by writing all power series appearing in multiplications of elements $d_{i}, d_{j}$ in $\mathcal N$ in variable $\bar t=t-t_{0}$,  we obtain 
 a subalgebra of $\mathcal N$ which can be extended to a deformation of $B$ at $\bar t=0$, and which is a strongly flat deformation to $B$ (by using variable $\bar t$). Observe that  using the variable $\bar t$ instead of $t$ in multiplication  does not change the type of algebra $\mathcal N$. It follows that $\mathcal N$ has the same type as $B$. The same argument applies to algebra $C$. It then follows that $A,B, C$ have the same type as $\mathbb C$-algebras and since they are semisimple $\mathbb C$-algebras they are all isomorphic $\mathbb C$-algebras. 

$ $

We will now proceed with the proof.

$ $

Without restricting generality we can assume that $t_{0}<t_{1}$. 
 By Corollary \ref{west} $\mathcal N$ gives a strongly flat deformation from $N$ to some semisimple $\mathbb C$-algebra $A$. 

$ $

{\em Part $1$.} In this part we will show that $\bar {\mathcal N}$ has the same type as $A$ (defined as in Definition \ref{type}).    Let $c$ be as in Definition \ref{flat}, so $t_{1}\in (0,c)$. Let ${\mathcal N}^{c}$ be defined as ${\bar {\mathcal N}}$ but instead of 
 ${\bar {\mathbb C}}\{t\}$ using ${{\mathbb C}\{t\}}_{c}$-the algebra of power series convergent at $[0, c)$.

Let $N, {\mathcal N}, {\bar {\mathbb C}}\{t\}, E,  g_{i,j,l}(t), $
 $d_{1}, \ldots , d_{n}$ be as in Definition \ref{76}, so 
  \[d_{i}\circ _{t}d_{j}=\sum_{l=0}^{n }g_{i,j,l}(t)d_{l},\]
 for some $g_{i,j,l}(t)\in {\mathbb C}\{t\}$. Note that if $\mathcal N$ gives a flat deformation for $t\in [0,c)$ for some real number $c$ (in the sense of Definition \ref{flat}) then $g_{i,j,l}(t)$ are convergent power series at $t\in [0,c)$ for some real number  $0<c$ (and hence strongly convergent for $t\in [0,c)$).

Let $S$ be the $\mathbb C$-algebra generated by power series 
 $g_{i,j,l}(t)$. Notice that all power series belonging to $S$ are convergent at $[0,c)$ since the power series $g_{i,j,l}(t)$ are convergent at $[0,c)$.

 Let ${\mathcal N}_{1}$ be the subalgebra of $\mathcal N$ which is spanned as $S$-linear space by $d_{1}, \ldots , d_{n}$. Then ${\mathcal N}_{1}$ has a type $c_{1}, c_{2} \ldots $
 for some nonnegative integers $c_{1}, c_{2}, \ldots $ (as defined in Definition \ref{type}). 
 Observe, that the ${\bar {\mathcal N}}$ has the same type $c_{1}, c_{2}, \ldots $,  and consequently the same type as $E$ (since $\mathcal N$ was obtained from ${\bar {\mathcal N}}$ by extending the scalars).
 
$ $

{\em Part $2$. }  In this part we will show that $\bar {\mathcal N}$ has the same type as $B$ (since $\mathcal N$ specialised for $t$ at some interval $(s-\epsilon , s+\epsilon )$ gives algebras isomorphic to $B$ for some real number $\epsilon >0$).    Let $s\in (0,c)$, then we can express power series $g_{i,j,l}(t)$ in variable $\bar t=t-s$. Then 
 $g_{i,j,l}(t)=f_{i,j,t}({\bar t})$.
 Note that $f_{i,j,l}({\bar t})$ are convergent at the interval $(-e,e)$ for some $e$ (since they
 are obtained from $g_{i,j,l}(t)$ by writing them in variable ${\bar t}$).
 Denote $S_{2}$ to be the  subalgebra of the algebra $\mathbb C\{{\bar t}\}$ generated by polynomials $f_{i,j,l}({\bar t})$. 
 Let ${\mathcal N}_{2}$ be the $S_{2}$-algebra spanned by elements $d_{1}, \ldots , d_{n}$ (as a linear $S$-space) and with the multiplication 
\[d_{i}\circ _{\bar t} d_{j}=\sum_{l=0}^{n }f_{i,j,l}({\bar t})d_{l}.\]
 Then ${\mathcal N}_{2}$ has the same type as ${\mathcal N}_{1}$ (because is obtained by writing it in a variable $
\bar t$ instead of $t$).

Let $M$ be the ${\mathbb C}\{\bar t\}$ algebra obtained by extending the multiplication 
  \[d_{i}\circ _{\bar t} d_{j}=\sum_{l=0}^{n }d_{l}f_{i,j,l}({\bar t})\]
 to obtain the ${ {\mathbb C}}\{\bar t\}$ algebra.  Then specialising at at ${\bar t}=0$ we obtain an  algebra isomorphic 
 to $B$, provided that $\mathcal N$ specialised at $t=s$ gives algebra $B$.
  Similarly, for $\varepsilon \in  (0,e)$ by specialising at ${\bar t}=\varepsilon $ we obtain algebra isomorphic to $B$, since $\mathcal N$ specialised at $t=s+\varepsilon $ gives an algebra isomorphic to $B$, by Gabriel's Proposition $2.5$ \cite{Gabriel} (as mentioned in the comment by Joachim Jelisiejew in the introduction).  

  Therefore, $M$ is a formal deformation of a $\mathbb C$-algebra isomorphic to $B$. Notice that because power series $f_{i,j,l}({\bar t})$ are convergent for $t\in (-e,e)$ for some integer $e>0$ then  we can consider ${\bar {\mathbb C}}\{\bar t\}$ -the ring of  power series convergent at $\bar t$ in some interval near $0$, and obtain a subalgebra $\bar M$ of $M$ which is spanned by $d_{1}, \ldots , d_{n}$ as an algebra over ${\bar {\mathbb C}}\{\bar t\}$.  This algebra can be then embedded into algebra $E'$ over the field of convergent Laurent series in variable ${\bar t}$. By Proposition \ref{polynomialidentity} $E'$ has the same type as $B$. Notice also that $E'$ has the same type as $M$, and hence the same type as ${\mathcal N}_{2}$.
%

 Recall that  ${\mathcal N}_{2}$ has the same type as ${\mathcal N}_{1}$ as shown before, and therefore has the same type as $A$ by Proposition \ref{polynomialidentity} (because $\mathcal N$ is a strongly flat deformation from $N$ to $A$ in variable $t$).
 In conclusion $A$ and $B$ have the same type, and hence they are $\mathbb C$-algebras, hence it follows that they are isomorphic. Similarly, $A$ and $C$ are isomorphic as $\mathbb C$-algebras.
  Therefore, $B$ and $C$ are isomorphic.
\end{proof}
\begin{corollary}

 Let $N, B$ be two finite-dimensional $\mathbb C$-algebras of the same dimension $n$ with $B$ semisimple and let $\mathcal N$ be a formal deformation of $N$.   Suppose that $\mathcal N$ is a flat deformation of $N$ (in the sense of Definition \ref{flat}) which deforms to $B$. Then $\mathcal N$ is a strongly flat deformation from $N$ to $B$.
\end{corollary}

{\bf Acknowledgments.} The author acknowledges support from the
EPSRC programme grant EP/R034826/1 
 and from the EPSRC research grant EP/V008129/1.  The author is very grateful to Joachim Jelisijew for many helpful comments, especially for his remark which was included in the introduction. The author is  also very grateful to Michael Wemyss for suggesting open questions. The author is  very grateful to Uzi Vishne for providing answers related to Remark \ref{commutative}. The author would like to thank  Alicja Smoktunowicz for many helpful comments which improved the paper. Many thanks to Michael West for his help with the English language aspects of
 the paper.

\end{document}